\documentclass[12pt]{amsart}

\usepackage{amssymb}
\usepackage{amsfonts}
\usepackage{graphicx}

\newcommand{\C}{\mathbf C}

\newcommand{\Rq}{{\mathbf R}^{4}}
\newcommand{\Rt}{{\mathbf R}^{3}}
\newcommand{\Rd}{{\mathbf R}^{2}}
\newcommand{\Cd}{{\mathbf C}^{2}}

\newcommand{\de}{\text{\rm d}\mspace{1mu}}

\newcommand{\fff}{\boldsymbol{I}}
\newcommand{\sff}{\mathbf{I} \mathbf{I}}
\newcommand{\He}{\text{Hess}\mspace{1mu}}
\newcommand{\rank}{\rm rank\,}
\newcommand{\beq}{\begin{equation}}
\newcommand{\eeq}{\end{equation}}
\newcommand{\ba}{\begin{align}}
\newcommand{\ea}{\end{align}}
\newcommand{\baa}{\begin{align*}}
\newcommand{\eaa}{\end{align*}}

\newtheorem{theorem}{Theorem}
\newtheorem{lemma}{Lemma}
\newtheorem{cor}{Corollary}
\newtheorem{prop}{Proposition}

\newtheorem*{wineq}{Wintgen inequality}
\newtheorem*{lnf}{Local Normal Form}

\newtheorem*{claim}{Claim}

\theoremstyle{definition}
\newtheorem{defn}{Definition}
\newtheorem{ex}{Example}

\theoremstyle{remark}
\newtheorem{rem}{Remark}

\begin{document}

\title[Lagrangean surfaces]{Inflection points and asymptotic lines on Lagrangean surfaces}

\author{J. Basto-Gon\c calves}
\date{\today}
\address{Centro de Matem\'atica da Universidade do
Porto, Portugal}
\email{jbg@fc.up.pt}
\thanks{Financial support from FCT, and  Calouste Gulbenkian
Foundation}


\begin{abstract}
We describe the structure of the asymptotic lines near an inflection point of a Lagrangean surface, proving that in the generic situation it corresponds to two of the three possible cases when the discriminant curve has a cusp singularity. Besides being stable in general, inflection points are proved to exist on a compact Lagrangean surface whenever its Euler characteristic does not vanish.
\end{abstract}
\maketitle

\section{Introduction}
The origin of the study of surfaces  in $\Rq$ was the interpretation of a complex plane curve as a real surface. This is very natural, as formulae for curves in the real plane, curvature for instance, can be adapted to complex curves and then seen again in a real version but for surfaces in $\Rq$.

In a different context, an invariant torus for a two degrees of freedom integrable Hamiltonian system is also a Lagrangean surface  in $\Rq$, and in general Lagrangean surfaces play an important role for 4-dimensional symplectic manifolds. In particular, the real version of plane complex curves are always Lagrangean surfaces for a suitable symplectic form.

From the point of view of singularity theory, and after the generic points, the first interesting points are inflection points. We follow a similar approach here, along the lines of \cite{little} and \cite{dcc,rdcc}, and consider the generic situation for inflection points in Lagrangean surfaces: the normal form at those points and thestructure of the asymptotic lines around them.

The generic picture for Lagrangean surfaces is quite different from that for arbitrary generic surfaces in $\Rq$: we prove that all inflection points are flat inflection points (this is a codimension one situation in general), and that there are two possible (topological) phase portraits for the asymptotic lines around an inflection point, corresponding to two of three cases considered in \cite{tari08} when the discriminant curve has a cusp singularity (this is  codimension two in general).

The presence of flat inflection points is a stable phenomenon for Lagrangean surfaces, and we prove their existence for compact Lagrangean surfaces with non vanishing Euler characteristic, giving an estimate of their number in the generic case.

\section{Moving frames}

We consider a surface $S\subset \mathbf R^4$ locally given by a parametrization:
\[
\Xi : U\subset\mathbf R^2\longrightarrow \Rq
\]
and  a set $\{e_1,e_2,e_3,e_4\}$  of orthonormal vectors, depending on $(x,y)\in U$,  satisfying:
\begin{itemize}
\item $e_1(x,y)$ and $e_2(x,y)$ span the tangent space $T_{\Xi(x,y)}S$ of $S$ at $\Xi(x,y)$.
\item $e_3(x,y)$ and $e_4(x,y)$ span the normal space $N_{\Xi(x,y)}S$ of $S$ at $\Xi(x,y)$.
\end{itemize}
Then $\Xi,\{e_1,e_2,e_3,e_4\}$ is an \emph{adapted moving frame} for $S$. Associated to this frame, there is a dual basis for 1-forms, $\{\omega_1,\omega_2,\omega_3,\omega_4\}$.

If we take $U$ small enough, $\Xi$ can be assumed to be an embedding; then the vectors $e_i$ and the 1-forms  $\omega_i$ can be extended to an open subset of $\Rq$.

While the image of $D\Xi$ is the tangent space of $S$, the image of the second derivative $D^2\Xi$ has both tangent and normal components; the vector valued quadratic form associated to the normal component:
\beq
(D^2\Xi\cdot e_3) e_3+(D^2\Xi\cdot e_4) e_4
\eeq
is the \emph{second fundamental form} $\sff$ of $S$. It can be written \cite{little} as $\sff_1e_3+\sff_2 e_4$, where:
\begin{subequations}\label{sff}
\begin{align}
\sff_1=&a\mspace{1mu} \omega_1^2+2b\mspace{1mu}  \omega_1  \omega_2+c\mspace{1mu}  \omega_2^2\label{sff1}\\
\sff_2=&e\mspace{1mu}  \omega_1^2+2f\mspace{1mu}  \omega_1  \omega_2+g\mspace{1mu}  \omega_2^2\label{sff2}
\end{align}
\end{subequations}

Let  $\mathcal M_1$ and $\mathcal M_2$ be the matrices associated to the above quadratic forms:
\[
\mathcal M_1=\left[\begin{matrix}
a&b\\
b&c
\end{matrix}\right],\qquad
\mathcal M_2=\left[\begin{matrix}
e&f\\
f&g
\end{matrix}\right]
\]

The \emph{mean curvature} $\mathcal H$ is defined by:
\beq
\mathcal H=\dfrac{1}{2}\left(\mathcal H_1+\mathcal H_2\right),
\quad
\mathcal H_i=\hbox{Tr}\mspace{2mu}\mathcal M_i , \qquad i=1,2
\eeq
and  \emph{Gaussian curvature} $K$ is given by:
\beq
K=K_1+K_2 ,
\quad K_i= \det\mathcal M_i , \qquad i=1,2
\eeq

We can express the Gaussian, mean  and  the \emph{normal curvature} in terms of the coefficients of the second fundamental form \cite{little}:
\begin{align}\label{curvatures}
K=&(ac-b^2)+(eg-f^2)\\
\kappa=&(a-c)f-(e-g)b\notag\\
\mathcal H=&\dfrac{1}{2}(a+c)e_3+\dfrac{1}{2}(e+g)e_4\notag
\end{align}

We consider a surface $S$ locally given by a parametrisation:
\[
\Xi : (x,y)\mapsto (x,y,\varphi(x,y),\psi(x,y))
\]
where $\Phi=(\varphi, \psi)$ has vanishing first jet at the origin,  $j^1\Phi(0)=0$.

The vectors $T_1$ and $T_2$ span the tangent space of $S$ and the vectors $N_1$ and $N_2$ span the normal space:
\begin{align}
T_1=&\Xi_x=(1,0,\varphi_x,\psi_x),& & N_1=(-\varphi_x,-\varphi_y,1,0)\\
\notag
T_2=&\Xi_y=(0,1,\varphi_y,\psi_y), && N_2=(-\psi_x,-\psi_y,0,1)
\end{align}
the index $z$ standing for derivative with respect to $z$.

The induced metric in $S$ is given by the first fundamental form:
\[
\fff=E\de x^2+2F\de x\de y+G\de y^2
\]
where:
\[
E=T_1\cdot T_1,\quad F=T_1\cdot T_2,\quad G=T_2\cdot T_2,
\]

We define:
\[
\hat E=N_1\cdot N_1,\quad \hat F=N_1\cdot N_2,\quad \hat G=N_2\cdot N_2,
\]
and it is easy to verify that:
\[
\hat E  \hat G-\hat F^2=W, \quad\hbox{where \ } W=EG-F^2
\]

Now consider the orthonormal frame defined by:
\begin{align}
e_1=&\dfrac{1}{\sqrt{E}}T_1,& & e_2=\sqrt{\dfrac{1}{EW}}\left(ET_2-F T_1\right)\\
\notag
e_3=&\dfrac{1}{\sqrt {\hat E}}N_1, && e_4=\sqrt{\dfrac{1}{\hat EW}}\left(\hat EN_2-\hat F N_1\right)
\end{align}
It is easy to see that:
\begin{align}\label{abc}
a=&\dfrac{1}{E\sqrt{\hat E}}\varphi_{xx}\\
\notag
b=&\dfrac{1}{E\sqrt{ W\hat E}}(E\varphi_{xy}-F\varphi_{xx})\\
\notag
c=&\dfrac{1}{EW\sqrt{\hat E}}(E^2\varphi_{yy}-2EF\varphi_{xy}+F^2\varphi_{xx})\\\label{efg}
e=&\dfrac{1}{E\sqrt{\hat E W}}\left(\hat E\psi_{xx}-\hat F\varphi_{xx}\right)\\
\notag
f=&\dfrac{1}{EW\sqrt{\hat E }}\left(E(\hat E\psi_{xy}-\hat F\varphi_{xy})-F(\hat E\psi_{xx}-\hat F\varphi_{xx})\right)\\
\notag
g=&\dfrac{1}{EW\sqrt{W\hat E }}\left(E^2(\hat E\psi_{yy}-\hat F\varphi_{yy})-\right.\\
\notag
&\phantom{\dfrac{1}{EW\sqrt{W\hat E }}}\left.-2EF(\hat E\psi_{xy}-\hat F\varphi_{xy})+F^2(\hat E\psi_{xx}-\hat F\varphi_{xx})\right)
\end{align}

Then, using these formul\ae\ or those from  \cite{am1,am2}, we obtain the following expressions for the Gaussian and normal curvature:

\begin{prop}
The Gaussian curvature is  given by:
\beq
K=\dfrac{1}{W^2}(\hat E H_\psi-\hat FQ+\hat G H_\varphi)\label{K}
\eeq
where:
\[
H_f=\He(f)=\left|
\begin{array}{cc}
f_{xx} &f_{xy}\\
f_{xy} & f_{yy}
\end{array}
\right|, \quad
Q=\left|
\begin{array}{cc}
\varphi_{xx} & \varphi_{xy}\\
\psi_{xy} & \psi_{yy}
\end{array}
\right|-\left|
\begin{array}{cc}
\varphi_{xy} & \varphi_{yy}\\
\psi_{xx} & \psi_{xy}
\end{array}\right|
\]
\end{prop}

\begin{prop}
The  normal curvature is given by:
\beq
\kappa=\dfrac{1}{W^2}(EL-FM+GN)\label{N}\
\eeq
where:
\[
L=\left|
\begin{array}{cc}
\varphi_{xy} & \varphi_{yy}\\
\psi_{xy} & \psi_{yy}
\end{array}
\right|, \quad
M=\left|
\begin{array}{cc}
\varphi_{xx} & \varphi_{yy}\\
\psi_{xx} & \psi_{yy}
\end{array}
\right|, \quad
N=\left|
\begin{array}{cc}
\varphi_{xx} & \varphi_{xy}\\
\psi_{xx} & \psi_{xy}
\end{array}
\right|
\]
\end{prop}

\section{Asymptotic directions}\label{Asymptotic directions}

The \emph{curvature ellipse} or \emph{indicatrix} $\mathcal E$ of the surface $S$ is the image under the second fundamental form of the unit circle in the tangent space:
\[
\mathcal E_p=\{v\in N_pS \ | \ v=\sff(u), \ u\in T_pS, \ |u|=1\}
\]

\begin{prop}\cite{little}
The normal curvature $\kappa$ is related to the oriented area $A$ of the curvature ellipse by:
\beq
\dfrac{\pi}{2}\kappa=A
\eeq
\end{prop}

The curvature ellipse at a point  $p\in S$ can be used to  characterize that point; in particular:
\begin{itemize}
\item  $p$ is a \emph{circle point} if the curvature ellipse at $p$ is a circumference.
\item $p$ is a \emph{minimal point} if the curvature ellipse at $p$ is centred at the origin, $\mathcal H(p)=0$.
\item $p$ is an \emph{umbilic point} if the curvature ellipse at $p$ is a circumference centred at the origin; the point is both a minimal and a circle point.
\end{itemize}

By identifying $p$ with the origin of $N_p S$,  the points of $S$ may be classified according to their position with respect to the curvature ellipse, that we assume to be non degenerate ($\kappa(p)\ne 0$), as follows:

\begin{itemize}
\item  $p$ lies outside the curvature ellipse.

\noindent
The point is said to be a
\emph{hyperbolic point} of $S$. The \emph{asymptotic directions} are the tangent directions whose images span the two normal lines tangent to the indicatrix passing through the origin;   the \emph{binormals} are the normal directions perpendicular  to those normal lines.

\item $p$ lies inside the curvature ellipse.

\noindent
The point $p$  is an \emph{elliptic point}. There are no binormals and no asymptotic directions.

\item $p$ lies on the curvature ellipse.

\noindent
The point $p$ is a \emph{parabolic point}. There is  one binormal and one asymptotic direction.

\end{itemize}

The points $p$ where the curvature ellipse passes through the origin are characterised by $\Delta(p)=0$~\cite{little}, where:
\beq
\Delta=\dfrac{1}{4}\left|
\begin{matrix}
a&2b&c&0\\
e&2f&g&0\\
0&a&2b&c\\
0&e&2f&g
\end{matrix}
\right|
\eeq
In fact, $\Delta$ is the resultant of the two polynomials $ax^2+2bxy+cy^2$ and $ex^2+2fxy+gy^2$. If $\sff(u)=0$ those polynomials have a common root $(u_1,u_2)$, and their resultant has to be zero.

The points of $S$ may be classified using $\Delta$, as follows~\cite{little}:

\begin{prop}
If a pont $p\in S$ is hyperbolic, parabolic or elliptic then  $\Delta(p) < 0$, $\Delta(p)=0$ or $\Delta(p) > 0$, respectively.
\end{prop}

We can  extend the definition of hyperbolic point, respectively elliptic point and parabolic point, to the case where $\kappa(p)=0$ by means of $\Delta$, as $\Delta(p)<0$, respectively $\Delta(p)>0$ and $\Delta(p)=0$.

\begin{defn}
The second-order \emph{osculating space} of the surface $S$ at $p\in S$ is the space generated by all vectors $\gamma'(0)$ and $\gamma''(0)$ where $\gamma$ is a curve through $p$ parametrized by arc length from $p$. An \emph{inflection point} is a point where the dimension of the osculating space is not maximal.
\end{defn}

\begin{prop}\cite{little} The following conditions are equivalent:
\begin{itemize}
\item $p\in S$ is an inflection point.
\item $p\in S$ is a point of intersection of $\Delta=0$ and $\kappa=0$.
\item ${\rank}\mathcal M(p) \le 1$
\end{itemize}
The inflection points are  singular points of $\Delta=0$.
\end{prop}

\begin{prop}\cite{dcc}
Let $p\in S$ be a generic inflection point. Then $p$ is a Morse singular point of $\Delta=0$,  and the Hessian $H_{\Delta}$ of $\Delta$ at $p$ has the same sign as the curvature $K(p)$.
\end{prop}

When $\Delta(p)=0$ we can distinguish among the following possibilities:

\begin{itemize}
\item $\Delta(p)=0$, $K(p) < 0 $ and ${\rank}\mathcal M(p) = 2$

\noindent
The curvature  ellipse is non-degenerate, $\kappa(p)\ne 0$; the binormal is the normal  at $p$.

\item $\Delta(p) = 0$, $K(p) < 0 $ and ${\rank}\mathcal M(p) = 1$

\noindent
$p$ is an \emph{inflection point
of real type}: the curvature ellipse is a radial segment and $p$ does not belong to it, $\kappa(p)= 0$. The point $p$ is a self-intersection point of $\Delta=0$, as $H_{\Delta}(p)<0$.
\item $\Delta(p) = 0$, $K(p)=0$

\noindent
$p$ is an \emph{inflection point of flat type}: the curvature ellipse is a radial segment and $p$  belongs to its boundary, $\kappa(p)= 0$.
\item $\Delta(p) = 0$, $K(p) > 0$

\noindent
$p$ is an \emph{inflection point of imaginary type}: the curvature ellipse is a radial segment and $p$ belongs to its interior, $\kappa(p)= 0$. The point $p$ is an isolated point of $\Delta=0$, as $H_{\Delta}(p)>0$.

\end{itemize}

At an inflection point the normal line perpendicular to the line through the origin containing the radial segment defines the binormal.

\begin{rem}
If $\Delta(p) = 0$ and $K(p)\ge 0$ then ${\rank}\mathcal M(p)\le 1$.
\end{rem}

\begin{rem}\label{flat}
It can be shown that for an open and dense set of embeddings of $S$ in $\mathbf R^4$, $\Delta^{-1}(0) \cup
K^{-1}(0) = \emptyset$; therefore on a generic surface there are no inflection points of flat type.
\end{rem}

The \emph{height function} on $S$ corresponding to $b\in\mathbf R^4$ is the map  $f_b(p)=f(p,b)$, where:
\[
f:S\times \mathbf R^4\longrightarrow\mathbf R,\quad f(p,b)=p\cdot b
\]

\begin{prop}\cite{dcc}
The critical points of $f$ are exactly the points of the normal space $NS$. Moreover:
\begin{itemize}
\item If $\Delta(p)>0$, then $f_b(p)=f(p,b)$ has a non degenerate critical point  at $p$ for all $b\in N_pS$.
\item If $\Delta(p)<0$, then $f_b(p)=f(p,b)$ has a   degenerate critical point at $p$ for exactly two independent normal directions.
\item If $\Delta(p)=0$, then $f_b(p)=f(p,b)$ has a  degenerate critical point at $p$  for exactly one normal direction.
\end{itemize}
\end{prop}

The surface $S$ has a higher order contact with the hyperplane normal to $b$ containing the tangent plane to $S$ at $p$, and as remarked in \cite{dcc} this shows that the binormal for a surface in $\Rq$ is an analogue to the binormal of a curve in $\Rt$. 

If the height function $f_b$ has a critical point at $p$, then $f_{\lambda b}$, with $\lambda>0$, has the same type of singular point at $p$; we will consider therefore the height map as being defined on $\mathbf S^3$:
\[
f:S\times \mathbf S^3\longrightarrow\mathbf R,\quad f(p,b)=f_b(p)=p\cdot b
\]
The critical points of $f$ are the points of the unit normal space $N^1S=\{(p,b)\mid p\in S, b\in N_pS, |b|=1\}$. 
If $f_b$ has a degenerate critical point at $p$, in general the kernel of  its second derivative defines a direction in the tangent space $T_pS$ (with the usual identifications), and that is an asymptotic direction:

\begin{prop}\label{binas}
Let $p\in S$ be a degenerate critical point of $f_b$ for some $b\in N_pS$,  $|b|=1$. If the kernel of $D^2f_b$ is one dimensional, it defines an asymptotic direction and $b$ defines a binormal.
\end{prop}

\begin{prop}\cite{dcc}
Let $p\in S$ be a parabolic point. If $p$ is not an  inflection point then: 
\begin{itemize}
\item $p$ is a fold singularity  of the height function when the asymptotic direction is not tangent to the line $\Delta^{-1}(0)$ of parabolic points.
\item $p$ is a cusp (or higher order) singularity when the asymptotic direction is  tangent to the line $\Delta^{-1}(0)$ of parabolic points.
\end{itemize}
\end{prop}

\begin{prop}\cite{dcc}
The inflection points of a surface correspond to umbilic singularities, or higher singularities, of the height function.
\end{prop}

The singularities of the family of height functions on a generic surface can be used \cite{dcc} to characterize the different points of that surface:
\begin{itemize}
\item An elliptic point $p$ is a nondegenerate critical point for any 
of the height functions associated to normal directions to $S$ at $p$. 
\item If $p$ is a hyperbolic point, there are exactly 2 normal directions 
 at $p$ such that $p$ is a degenerate critical point of their corresponding 
height functions. 
\item If $p$ is a parabolic point, there is a unique normal direction such that 
$f_b$ is degenerate at $p$.
\begin{itemize}
\item A parabolic point  $p$ is  a fold singularity of $f_b$ if and only if  the unique asymptotic direction is not tangent to the line of parabolic points $\Delta^{-1}(0)$. 
\item A parabolic point  $p$ is  a cusp singularity of $f_b$ if and only if  $p$ is a \emph{parabolic cusp} of 
$S$, where the asymptotic direction is tangent to the line of parabolic points. 
\item A parabolic point  $p$ is  an umbilic point for $f_b$ if and only if  $p$ is an inflection point of 
$S$. 
\end{itemize}
\end{itemize}

\begin{rem}
For a generic surface, the points $p$ which are a swallowtail singularity of $f_b$ do not belong to the line of parabolic points; at a swallowtail singularity one of the asymptotic directions is tangent to line of points where $f_b$ has a cusp singularity.
\end{rem}

\

\section{Lagrangean surfaces}

A \emph{symplectic manifold} is a pair $(M,\omega)$, where $M$ is a $2n-$dimensional differentiable manifold and $\omega$ is a symplectic form: a closed non degenerate $2$-form. Then:
\[
\Omega=\dfrac{1}{n!}\,\omega^n \hbox{ \  is a volume form, and } \de\mspace{1mu}\omega=0
\]
A {\sl symplectic map} is  a map $\varphi:(M,\omega)\longrightarrow (M',\omega')$, such that:
\[
\varphi^*\omega' =\omega
\]

A  \emph{Lagrangean submanifold} $L$ of $(M,\omega)$ is an immersed submanifold of $M$ such that:
\[
i^* \omega\equiv 0,\quad \hbox{where }i:L\longrightarrow L \hbox{ is the immersion map}
\]

\begin{defn}
A \emph{Lagrangean surface} $\mathcal L$ is an immersed Lagrangean submanifold of $(\Rq,\omega)$.\end{defn}

We consider $\Rq$ with the standard inner product and metric, and also with the standard symplectic form $\omega$:
\[
\omega= 
\de x\wedge \de u+\de y\wedge \de v
\]
in the coordinates $(x,y,u,v)$.

\begin{ex}
[Whitney sphere]
A Whitney sphere is a
Lagrangian immersion of the unit sphere $\mathbf S^2$, centered at the origin of $\mathbf R^3$, 
in $\mathbf R^4$ given by:
\[
\Phi(x,y,z ) = \dfrac{r}
{1 +z^2} (x , xz , y, yz) + C,\quad x^2+y^2+z^2=1
\] 
where $r$ is a positive number and $C$ is a vector of $\mathbf R^4$,  respectively the radius and the centre of the Whitney sphere.

The Whitney spheres  are embedded except at the poles of $\mathbf S^2$, where they 
have  double points. 
\end{ex}

We recall some results from \cite{eu}:

\begin{lnf}
Given $p\in \mathcal L$, there is a  change of coordinates, by a translation and a linear symplectic and orthogonal map, such that locally $\mathcal L$ becomes the graph around the origin of a map 
\[
\Phi=(\varphi,\psi): \Rd \longrightarrow\Rd
\]
satisfying:
\begin{itemize}
\item The first jet of $\Phi$  is zero at the origin.
\item $\dfrac{\partial\varphi}{\partial y}\equiv\dfrac{\partial\psi}{\partial x}$
\end{itemize}
\end{lnf}

\begin{rem}
If we preserve orientation, so that the linear map $A\in SO(4)$, there is another normal form; the symplectic form in $\Rq$ is then $\omega'=
\de x\wedge \de u-\de y\wedge \de v$ and the identity in the normal form is:
\[
\dfrac{\partial\varphi}{\partial y}\equiv-\dfrac{\partial\psi}{\partial x}
\]
\end{rem}

\begin{prop}\cite{eu}
A necessary  condition for $\mathcal L\subset \Rq$ to be a Lagrangean surface is that the Gaussian curvature and the normal curvature coincide up to sign:
\[
|K|\equiv |\kappa|
\]
\end{prop}

\begin{rem}
If $\mathcal L\subset \Rq$ is a Lagrangean surface, in the moving frame associated to the local normal form we  have:
\beq
e\equiv b,\quad f\equiv c\label{e=b}
\eeq
\end{rem}

\begin{cor}\cite{eu}
Let $\mathcal L$ be a Lagrangean surface and $p\in \mathcal L$. Then the following conditions are equivalent:
\begin{enumerate}
\item $p$ is an inflection point.
\item $p$ is a parabolic point where the Gaussian curvature vanishes.
\item $p$ is a parabolic point where the normal curvature vanishes.
\item $rank\begin{pmatrix}
a&b&c\\
e&f&g
\end{pmatrix} \le 1$.
\item $\rank\begin{pmatrix}
a&b&e&f\\
b&c&f&g
\end{pmatrix} \le 1$.
\end{enumerate}
\end{cor}

\section{Asymptotic lines for Lagrangean surfaces}

There are exactly two asymptotic directions 
at a hyperbolic point and just one asymptotic direction at a parabolic point, unless 
it is an inflection point, in which case all the directions are asymptotic.

The  condition for  a tangent vector $T=\cos \theta \; T_1 +\sin \theta \;T_2$ to span an asymptotic direction is:
\[
(af-be) \cos ^2\theta+(ag-ce) \cos \theta  \sin \theta +(bg-cf) \sin ^2 \theta =0
\]
and with the natural identifications,  the asymptotic directions are solutions to the binary implicit differential equation:
\begin{equation}\label{lasym}
(af-be) \de x^2+(ag-ce)\de x\de y+(bg-cf)\de y^2=0
\end{equation}

We have seen before  (\ref{e=b}) that on a Lagrangean surface we have:
\[
e\equiv b,\quad f\equiv c\hbox{\ \ and thus\ \ }af-be=ac-b^2,\quad bg-cf=eg-f^2
\]
A straightforward computation gives:
\begin{align}\label{cal}
af-be=&\dfrac{1}{\sqrt{EW}}\phantom{(}H_{\varphi}\\\notag
ag-ce=&\dfrac{1}{\sqrt{EW}}\left(EM-2FH_{\varphi}\right)\\\notag
bg-cf=&\dfrac{1}{\sqrt{EW}}\left(E^2H_{\psi}- EFM+F^2H_{\varphi}\right)
\end{align}

The implicit differential equation  (\ref{lasym}) for  the asymptotic directions
becomes:
\beq
H_{\varphi} \de x^2+(EM-2FH_{\varphi})\de x\de y+(E^2H_{\psi}- EFM+F^2H_{\varphi})\de y^2=0
\eeq
with discriminant curve:
\[
\Delta =(ag-ce)^2-4(af-be)(bg-cf)=0
\]
given by:
\[
M^2-4H_{\varphi}H_{\psi}=0
\]

The inflection points for Lagrangean surfaces are quite special, in particular the Gaussian and normal curvatures vanish on them; the following results were proved in \cite{eu}:
 
\begin{prop}
If the Gaussian curvature vanishes at a point $p$  in the discriminant curve $\Delta=0$, then:
\begin{itemize}
\item  The discriminant curve has a non Morse singularity at $p$
\item The point $p$ is a flat inflection point
\end{itemize}
\end{prop}

\begin{theorem}
The existence of a flat inflection point in a generic Lagrangean surface is a stable situation: it persists for small perturbations inside the class of Lagrangean surfaces.
\end{theorem}

\subsection{Normal form around an inflection point}

Our interest in the inflection points is the study of the asymptotic lines around them, therefore we can make changes of coordinates that are  isometries and symplectomorphisms,  as they do not affect those objects.

\begin{prop}\label{nfip}
Let $p\in S$ be an inflection point in a generic Lagrangean surface. Then around $p$ the surface $S$ can be given as the graph around the origin of a map $(x,y)\mapsto (\varphi(x,y), \psi(x,y))$ such that:
\begin{align*}
j^3\varphi=&\dfrac{1}{2}\eta x^2+\dfrac{1}{3}\zeta_1x^3+\zeta_2x^2y+\zeta_3xy^2+\dfrac{1}{3}\zeta_4y^3\\
j^3\psi=&\phantom{\dfrac{1}{2}\eta x^2+}\dfrac{1}{3}\zeta_2x^3+\zeta_3x^2y+\zeta_4xy^2 + \dfrac{1}{3}\zeta_5y^3
\end{align*}
\end{prop}

\begin{proof}
As $H_{\varphi}(0,0) = H_{\psi}(0,0)=0$ we see that $S$ is the graph of $(x,y)\mapsto (\varphi(x,y), \psi(x,y))$, with:
\[
j^2\varphi=\dfrac{1}{2}A(a_1x+a_2y)^2, \qquad j^2\psi=\dfrac{1}{2}B(b_1x+b_2y)^2
\]
and $\|a\|=\|b\|=1$, where $a=(a_1,a_2)$, $b=(b_1,b_2)$.

From $K(0,0)=H_{\varphi}(0,0) + H_{\psi}(0,0)=0$ it follows that  $0=\kappa (0,0)=L(0,0)+N(0,0)$ and we also have:
\[
\left|
\begin{matrix}
\varphi_{xy}&\psi_{xy}\\
\varphi_{yy}&\psi_{yy}
\end{matrix}
\right|
+
\left|
\begin{matrix}
\varphi_{xx}&\psi_{xx}\\
\varphi_{xy}&\psi_{xy}
\end{matrix}
\right|=0
\]
at the origin. This is readily evaluated as:
\[
AB(a\cdot b)(a\wedge b)=0
\]

In the generic case $AB\neq 0$ and thus either $a=\pm b$ or $a\perp b$; if we assume $a\cdot b$, or $b_1=a_2$, $b_2=-a_1$, the conditions:
\[
\varphi_{xy}(0,0)=\psi_{xx}(0,0), \quad  \varphi_{yy}(0,0)=\psi_{xy}(0,0) 
\]
are equivalent to:
\[
Aa_1a_2=Ba_2^2, \quad Aa_2^2=-Ba_1a_2
\]
These are impossible if  we assume the generic condition $\varphi_{xy}(0,0)\neq 0$; under this condition we have $a\wedge b$, or $a=\pm b$ and also:
\[
B=A\dfrac{a_2}{a_1}
\]
so that:
\[
j^2\varphi=\dfrac{1}{2}A(a_1x+a_2y)^2, \qquad j^2\psi=\dfrac{1}{2}A\dfrac{a_2}{a_1}(a_1x+a_2y)^2
\]

The change of coordinates:
\begin{align*}
\hat x=a_1x+a_2y \quad  \hat y=&a_2x-a_1y\\
\hat  u=a_1u+a_2v \quad \hat \hat v=&a_2u-a_1v
\end{align*}
is a symplectomorphism and an isometry, and in these coordinates (but writing with the old variables) we get:
\[
j^2\varphi=\dfrac{1}{2}\eta x^2, \quad j^2\psi=0, \qquad \eta=A\left(a_1+\dfrac{a_2^2}{a_1}\right)
\]
\end{proof}

It is easy to see that this type of coordinate change takes solutions of the binary differential equation for the asymptotic directions in the old coordinates to solutions of the corresponding equation for the new coordinates.

\section{Binary differential equation for the asymptotic lines}

As we have seen before, the implicit differential equation for  the asymptotic directions is the binary differential equation:
\beq\label{bde}
H_{\varphi} \de x^2+(EM-2FH_{\varphi})\de x\de y+(E^2H_{\psi}- EFM+F^2H_{\varphi})\de y^2=0
\eeq
and its discriminant curve is given by:
\[
M^2-4H_{\varphi}H_{\psi}=0
\]

Our aim is to describe the structure of the asymptotic lines around a flat inflection point, assumed to be the origin $(x,y)=0$. This will depend only on the 2-jets of the coefficients of the binary differential equation. We refer to the binary differential equation whose coefficients are the $k$-jets of the coefficients of the original  equation as the $k$-jet of that equation.

\begin{prop}
The 2-jet of the binary differential equation (\ref{bde}) is the same as the 2-jet of the binary differential equation:
\beq\label{bder}
H_{\varphi} \de x^2+M\de x\de y+H_{\psi}\de y^2=0
\eeq
which has the same discriminant curve.
\end{prop}

We assume (prop.\ref{nfip}) that:
\[j^2\varphi=\dfrac{1}{2}cx^2, \qquad j^2\psi=0
\]

We denote by $F(x,y,[\de x:\de y])=0$ the above equation (\ref{bder}): the bivalued direction field it determines can be lifted to a univalued vector field $X$ on the equation surface:
\beq\label{bdes}
\mathcal E=\{(x,y,[\de x:\de y]) | F(x,y,[\de x:\de y])=0\}\subset \mathbf R^2\times  \mathbf {RP}^1
\eeq
It is easier to consider an affine chart on $\mathbf {RP}^1$ and then:
\beq\label{bdesa}
\mathcal E=\{(x,y,p) | F(x,y,p)=H_{\varphi} +M p+H_{\psi}p^2=0\}, \quad p=\dfrac{\de y}{\de x}
\eeq
The lifted vector field is:
\beq\label{lvf}
X=\dfrac{\partial F}{\partial p}\dfrac{\partial}{\partial x}+p\dfrac{\partial F}{\partial p}\dfrac{\partial}{\partial y}-\left(
\dfrac{\partial F}{\partial x}+p\dfrac{\partial F}{\partial y}\right )\dfrac{\partial}{\partial p}
\eeq

The criminant curve is:
\[
\mathcal C=\left\{(x,y,p) | F(x,y,p)=0,\ \dfrac{\partial F}{\partial p}(x,y,p)=0 \right\}
\]
and its projection on the plane $(x,y)$ is the discriminant curve. The critical points of the field of asymptotic lines we want to study are projections of the critical points of the lifted vector field $X$ on the criminant curve.

Similarly, we can consider another affine chart on $\mathbf {RP}^1$:
\beq\label{bdesaq}
\mathcal E=\{(x,y,p) | \hat F(x,y,q)=H_{\varphi} q^2+M q+H_{\psi}=0\}, \quad q=\dfrac{\de x}{\de y}
\eeq
The lifted vector field is:
\beq\label{lvf}
\hat X=q \dfrac{\partial \hat F}{\partial q}\dfrac{\partial}{\partial x}+\dfrac{\partial \hat F}{\partial q}\dfrac{\partial}{\partial y}-q \left(
\dfrac{\partial \hat F}{\partial x}+\dfrac{\partial \hat F}{\partial y}\right )\dfrac{\partial}{\partial q}
\eeq
The vector field $X$ and $\hat X$ span the same direction at every point, and have the same critical points of the same type (disregarding orientation).
 
The criminant curve is:
\[
\mathcal C=\left\{(x,y,q) | \hat F(x,y,q)=0,\ \dfrac{\partial \hat F}{\partial q}(x,y,q)=0 \right\}
\]
and its projection on the plane $(x,y)$ is the discriminant curve.

 Abusing notation, we will not distinguish $F$ and $\hat F$, or $X$ and $\hat X$, the context should make it clear what coordinates are used.

The $p$-axis, or the $q$-axis, belongs to the criminant curve, since at a flat inflection point we have:
\[
H_{\varphi}=0, \quad M=0, \quad H_{\psi}=0, \quad \hbox{thus\ \ }\dfrac{\partial F}{\partial p}(0,0,p)\equiv 0,\ \dfrac{\partial F}{\partial q}(0,0,q)\equiv 0
\]

It also follows from this  that the $p$-axis is invariant for the lifted vector field $X$;
on the $p$-axis,  $X$ reduces to the differential equation:
\beq\label{lvfp}
\dot p=-\Pi(p), \quad \Pi(p)=\left(\dfrac{\partial F}{\partial x}+p\dfrac{\partial F}{\partial y}\right)\biggr |_{(x,y)=0}
\eeq
It has, in general, one or three critical points, as $\Pi(p)$ is a cubic polynomial in $p$:
\[
\Pi(p)=Ap^3+Bp^2+Cp+D
\]
where:
\[
A=H_{\psi,y},\quad B=H_{\psi,x}+M_y, \quad C=M_x+H_{\varphi,y}, \quad D=H_{\varphi,x}
\]
are computed at the origin.

From the normal form it follows that:
\[
\Pi(p)=\eta (\zeta_5 p^2+2\zeta_4 p+\zeta_3)
\]
and thus $p=\infty$ is always a critical point. It will be easier to use (\ref{bdesaq}), then:
\[
\dot q=-\hat \Pi(q),\qquad \hat \Pi(q)=\eta (\zeta_5 +2\zeta_4 q+\zeta_3 q^2) q
\]
and now the critical point will be at $q=0$. As before, we will not distinguish $\Pi$ and $\hat \Pi$.

The \emph{genericity condition} assumed throughout is that the roots of $\Pi(p)$ be simple. This is equivalent to:
\beq\label{transcond}
\eta\ne 0, \quad \zeta_5\ne 0, \quad \zeta_4^2-\zeta_3\zeta_5\ne 0
\eeq

\begin{prop}
In the generic case, the number and nature of the critical points of $X$ depends only on the 1-jet of the binary differential equation (\ref{bder}).
\end{prop}
\begin{proof}
We consider the lifted vector field $X$ on $\mathbf R^3$; we are interested on its critical points on the $p$-axis. They correspond to the zeros of $\Pi(p)$, and therefore there are one or three critical points.

The linear part of $X$ has always a zero eigenvalue $\mu_2$, since the two first lines are linearly dependent, and the eigenvalue corresponding to the $p$-axis is $\mu_1=-\Pi'(p)$; also $\mu_1+\mu_2+\mu_3$ is the trace of the linear part, $-F_{y}(0,0,p)$.
So the eigenvalues of $X$  are:
 \begin{align}
 \mu_1&=-F_{xp}-pF_{yp}-F_{y}\big |_{(0,0,p)}\\
 \notag
\mu_2&=0\\
\notag
\mu_3&=F_{xp}+pF_{yp}\big |_{(0,0,p)}
\end{align}

From $\Pi(p)=\eta (\zeta_5 p^2+2\zeta_4 p+\zeta_3)$ it follows that there exists a unique critical point when $\zeta_4-\zeta_3\zeta_5<0$, and three critical points when $\zeta_4-\zeta_3\zeta_5>0$.
\end{proof}

\begin{defn}
A singular point $P$ in a  topological surface $S$ is a saddle-type singularity for a  continuous vector field $Y$ on $S$ if:
\begin{itemize}
\item $P$ is an isolated singularity of $S$ and $Y$.
\item $S-P$ is a smooth surface, and $Y$ is smooth on $S-P$
\item There exist two smooth invariant curves on $S$ crossing transversally at $P$.
\item The invariant curves are the stable and unstable manifols of $Y$.
\end{itemize}
\end{defn}

\begin{theorem}
For a generic Lagrangean surface,  the  surface  $\mathcal E$, corresponding to the binary differential equation, is smooth  except at the  point $P_S=(0,0,[0:1])$. At that point:
\begin{itemize}
\item $\mathcal E$ has a Morse cone-like singularity.
\item The lifted vector field $X$ has a saddle-type singularity.
\end{itemize}
\end{theorem}
\begin{proof}

We have:
\[
H_{\psi,x}=0, \quad
H_{\psi,y}=0, \quad
M_x=H_{\varphi,y}
\]
at $x=y=0$, and therefore we obtain,  using (\ref{bdesaq}):
\[
\dfrac{\partial F}{\partial x}(0,0,q)_{q=0}=0, \quad \dfrac{\partial F}{\partial y}(0,0,q)_{q=0}=0
\]

Since
\[
\dfrac{\partial F}{\partial q}(0,0,q)\equiv 0
\]
it follows that surface equation $\mathcal E$ has a  singularity at the  point $P_S=(0,0,[0:1])$; the same is true for the lifted vector field $X$, as:
\[
\dot q=- \Pi(q)=-q\dfrac{\partial F}{\partial x}(0,0,q)- \dfrac{\partial F}{\partial y}(0,0,q)=0
\]
which vanishes for $q=0$.

\begin{claim}
The  Hessian matrix of $F$  with respect to $(x,y)$ at $P_S=(0,0,[0:1])$ depends only on the 3-jet of $\varphi$ and $\psi$.
\end{claim}

 Let:
\[
F_2=\left[
\begin{matrix}
F_{xx} & F_{xy}\\
F_{xy} & F_{yy}
 \end{matrix}\right]_{(0,0)}=
\left[ \begin{matrix}
H_{\psi,xx} & H_{\psi,xy}\\
H_{\psi,xy} & H_{\psi,yy}
 \end{matrix}\right]_{(0,0)}
\]

Using again the normal form we get:
\begin{align*}
H_{\psi,xx}(0,0)=&2(\zeta_2 \zeta_4-\zeta_3^2)\\
H_{\psi,xy}(0,0)=&\zeta_2 \zeta_5-\zeta_3\zeta_4\\
H_{\psi,yy}(0,0)=&2(\zeta_3 \zeta_5-\zeta_4^2)
\end{align*}
and the claim is proved.

\begin{claim}
$P_S=(0,0,[0:1])$ is a Morse singularity of the surface $\mathcal E$.
\end{claim}

We need to prove that in general the determinant of the Hessian matrix of $F$ is non zero. Since $F_{qq}(0,0,0)=0$ we have:
\begin{align}
H_F(0,0,0)=&
\left|
\begin{matrix}
F_{xx} & F_{xy}& F_{xq}\\
F_{xy} & F_{yy}& F_{yq}\\
 F_{xq}& F_{yq}&0
 \end{matrix}\right|
 =2 F_{xy}F_{yq} F_{xq}-F_{xx}F_{yq}^2-F_{yy}F_{xq}^2\\
 \notag
 =&-Q_2(F_{yq},-F_{xq})=-Q_2(\mathfrak f)
\end{align}
where we consider $Q_2$ as the quadratic form with matrix $F_2$, and:
\[
\mathfrak f=\left[ \begin{matrix}
F_{yq}\\
- F_{xq}
 \end{matrix}\right]=
 \left[ \begin{matrix}
M_{y}\\
- M_{x}
 \end{matrix}\right]=\eta
  \left[ \begin{matrix}
\zeta_5\\
-\zeta_4
 \end{matrix}\right]
\]

It is easy to see that $F_{xq}$ and $F_{yq}$  at $P_S$ depend only on the 3-jet of $\varphi$ and $\psi$, and the same is true for $F_{xx}$, $F_{xy}$ and $F_{yy}$ according to the previous claim.

The 2-jet of $\varphi$ and $\psi$ at $(0,0)$ is completely determined by  $\eta=\varphi_{xx}(0,0)$, which can be chosen arbitrarily apart from the genericity condition $\eta\ne 0$. Note that the 1-jets are zero.

The 3-jet of $\varphi$ and $\psi$ at $(0,0)$ depends of an extra arbitrary choice of 5 variables, $\zeta_1$, \ldots ,$\zeta_5$.
Therefore $H_F(0,0,0)=0$ is an algebraic condition on the 6 variables $\eta$ and $\zeta$; moreover, that condition is not an identity, as if for instance we choose:
\[
\zeta_1=\zeta_2=\zeta_3=0,\quad \zeta_4=\zeta_5=1
\]
we have:
\[
F_2(0,0,0)=
\left[ \begin{matrix}
0 &  0\\
0 &  -2
 \end{matrix}\right],\quad 
 \mathfrak f=\eta\left[ \begin{matrix}
1\\
-1
 \end{matrix}\right]
\]
and:
\[
H_F(0,0,0)=2\eta^2\ne 0
\]

Then, in an arbitrary open neighbourhood of values  $\eta$ and $\zeta$ satisfying
$H_F(0,0,0)=0$ there are points for which $H_F(0,0,0)\ne 0$, or equivalently there exists a small perturbation of the 3-jets  of $\varphi$ and $\psi$ at $(0,0)$ for which  the Hessian matrix of $F$ at $P_S=((0,0,0)$ is  non degenerate.

\begin{claim}
$P_S=(0,0,[0:1])$ is a cone-like singularity  of the surface $\mathcal E$.
\end{claim}

Now it is enough to show that there exists a direction on which the quadratic form corresponding to the Hessian matrix of $F$ is  zero. The $q$-axis is one such direction:
\[
\left[
\begin{matrix}
0 & 0& q
 \end{matrix}\right]
\left[
\begin{matrix}
F_{xx} & F_{xy}& F_{xq}\\
F_{xy} & F_{yy}& F_{yq}\\
 F_{xq}& F_{yq}&0
 \end{matrix}\right]
 \left[
\begin{matrix}
0 \\ 0\\ q
 \end{matrix}\right]=
 \left[
\begin{matrix}
0 & 0& q
 \end{matrix}\right]
 \left[
\begin{matrix}
* & *& 0
 \end{matrix}\right]=0
\]

\begin{claim}
$P_S=(0,0,[0:1])$ is a saddle-type singularity for the lifted vector field $X$.
\end{claim}

We consider the vector field $X$ on $\mathbf R^3$, not its restriction to the equation surface $\mathcal E$. As seen before, $P_S=(0,0,0)$ is a singular point of $X$.

The linear part of $X$ at $P_S=(0,0,0)$ is given by:
\beq
A_X=\left[
\begin{matrix}
qF_{xq} & qF_{yq}& qF_{qq}+F_{q}\\
F_{xq} & F_{yq}& F_{qq}\\
 -qF_{xx}-F_{xy}&
 -qF_{xy}-F_{yy}&
 -qF_{xq}-F_{yq}-F_{x}
 \end{matrix}\right]
\eeq
computed at the singular point. Since:
\[
 -qF_{xq}-F_{yq}-F_{x}\big |_{(0,0,0)}=-\Pi'(0), \quad
F_{q}{(0,0,0)}=F_{qq}{(0,0,0)}=0
\]
we see that the $q$-axis is an eigendirection, and the corresponding eigenvalue is $-\Pi'(0)$.

It also follows that $0$ is an eigenvalue, as the first line is a multiple of the second one. From:
\[
\text{\rm Tr}\mspace{2 mu} A_X=-F_{x}, \quad F_{x}{(0,0,0)}=0
\]
the third eigenvalue is $\Pi'(0)$.

The  plane $x=0$ is invariant; the eigendirection corresponding to the $0$ eigenvalue does not belong to that plane, and the restriction of the linearized $X$ to  $x=0$ is a saddle with eigenvalues $\pm \Pi'(0)$.

The vector field $X$ on $\mathbf R^3$ has two separatrices at $P_S=(0,0,0)$, the stable and unstable manifolds, that are smooth curves and contained in the equation surface $\mathcal E$: $F$ is a first integral of $X$. These separatrices are invariant curves for the restriction of $X$ to $\mathcal E$.
\end{proof}

\begin{prop}
Assuming the genericity condition, the 1-jet of the  binary differential equation (\ref{bder}) can be reduced to the normal form:
\beq\label{nform1}
Y \de Y^2\pm 2X\de X\de Y=0
\eeq
\end{prop}
\begin{proof}
The 1-jet of the binary differential equation (\ref{bder}) is:
\[
\eta(\zeta_4x+\zeta_5y)\de x\de y+\eta(\zeta_3x+\zeta_4y)\de x^2=0
\]

Since $\eta\ne 0$ it can be omitted. Consider the linear change of coordinates:
\[
x=\alpha X+\beta Y, \quad y=\gamma X+\delta Y
\]
and the equation in these coordinates:
\[
(A_1X+A_2Y) \de Y^2+(B_1X+B_2Y)\de X\de Y+(C_1X+C_2Y)\de X^2=0
\]
Then:
\[
\left[\begin{matrix}A_1&A_2\\B_1&B_2\\C_1&C_2\end{matrix}\right]=
\left[\begin{matrix}
\delta^2&\beta\delta&\beta^2\\
2\gamma\delta&\alpha\delta+\beta\gamma&2\alpha\beta\\
\gamma^2&\alpha\gamma&\alpha^2\end{matrix}\right]
\left[\begin{matrix}0&0\\ \zeta_4&\zeta_5\\ \zeta_3&\zeta_4\end{matrix}\right]
\left[\begin{matrix}\alpha&\beta\\ \gamma&\delta\end{matrix}\right]
\]
and as:
\[
C_1=\alpha\gamma(\alpha \zeta_4+\gamma \zeta_5)+\alpha^2(\alpha \zeta_3+\gamma \zeta_4)
\]
if we take $\alpha=0$ we obtain $C_1=0$, and in fact $C_2=0$ as well.

Now:
\[
A_1=\beta\delta \gamma \zeta_5+\beta^2 \gamma  \zeta_4=\beta \gamma (\delta \zeta_5+\beta   \zeta_4)
\]
and if we take:
\[
\alpha=0, \quad \delta \zeta_5+\beta   \zeta_4=0
\]
we obtain $C_1=C_2=A_1=0$ and also:
\begin{align*}
A_2&=\beta^2( \delta \zeta_4+\beta   \zeta_3)\\
B_1&=\beta \gamma^2\zeta_5\\
B_2&=\beta\gamma(\beta   \zeta_4+  \delta \zeta_5)=0
\end{align*}

With $\beta$ different from zero, the coefficients can be divided by $\beta$.
We can choose $\gamma$ so that:
\[
B_1=\gamma^2\zeta_5=\pm 2, \quad \hbox{assuming }\zeta_5\ne 0
\]
and $\delta$ and $\beta$ so that:
\begin{align*}
\delta \zeta_5+\beta   \zeta_4&=0\\
\beta( \delta \zeta_4+\beta   \zeta_3)&=\pm 1, \quad \hbox{assuming }\zeta_4^2-\zeta_3\zeta_5\ne 0
\end{align*}
In fact, if $\zeta_4\ne 0$ we can solve the first condition for $\beta$ and upon substitution the second one becomes:
\[
-\dfrac{\zeta_5}{\zeta_4^2}(\zeta_4^2-\zeta_3\zeta_5)\delta^2=\pm 1
\]
which can be solved if  $\zeta_4^2-\zeta_3\zeta_5\ne 0$ leading to $\beta\ne 0$.

If  $\zeta_4=0$ the first condition gives $\delta=0$ and the second becomes:
\[
\beta^2   \zeta_3=\pm 1
\]
We can find such a $\beta\ne 0$ if $\zeta_3\ne 0$; but this is equivalent to $\zeta_4^2-\zeta_3\zeta_5\ne 0$ since $\zeta_5\ne 0$ and $\zeta_4=0$.

The genericity conditon that the zeros of $\Pi(q)=\eta(\zeta_5 +2\zeta_4 q+\zeta_3 q^2)q$ be simple translates into:
\[
\eta\ne 0, \quad \zeta_5\ne 0, \quad \zeta_4^2-\zeta_3\zeta_5\ne 0
\]
and the proof is finished.
\end{proof}

\begin{prop}
At a generic inflection point the discriminant curve has a cusp singularity.
The genericity condition that the root $q=0$ be simple is equivalent to a transversality condition: the direction $[0:1]$ is not tangent to the discriminant curve at the cusp point.
\end{prop}
\begin{proof}
The plane curve $\Delta (x,y)$ has a cusp singularity at a singular point $P$, that we take to be the origin,  if:
\begin{itemize}
\item The quadratic part of $\Delta$ at $P$ is a nonzero square $Q^2$
\item The cubic part  of $\Delta$ at $P$ is not divisible by $Q$.
\end{itemize}

We can  compute explicitly the second order terms $\Delta_2$ in $\Delta=M^2-4H_{\varphi}H_{\psi}$:
\[
\Delta_2(x,y)=\left[\eta (\zeta_4 x+\zeta_5y)\right]^2
\]
The condition of $\zeta_4 x+\zeta_5y$ dividing
the cubic terms is an algebraic condition on the $4$-jet of $\varphi$ and $\psi$, and not an identity, so $\zeta_4 x+\zeta_5y$ does not divide the cubic terms with the coefficients of that $4$-jet in an open dense set. The discriminant curve has a cusp point at the origin, and the tangent line there is defined by a vector $(x,y)$ such that $\Delta_2(x,y)=0$:
\[
\zeta_4 x+\zeta_5y=0,\quad \eta\ne 0
\]

The genericity condition above, the root $q=0$ being simple, is $\eta\ne 0$ and $\zeta_5\ne 0$, and this means that the direction $[0:1]$ is not tangent to the discriminant curve at the cusp point:
\[
\Delta_2(0,1)=\left[\eta \zeta_5\right]^2\ne 0
\]
\end{proof}

We consider the lifted vector field $X$ on $\mathbf R^3$; we are interested on its critical points on the $p$-axis. They correspond to the zeros of $\Pi(p)$, and therefore there are one or three critical points.

The eigenvalues of $X$  are:
 \begin{align}
 \mu_1&=-F_{xp}-pF_{yp}-F_{y}\big |_{(0,0,p)}\\
 \notag
\mu_2&=0\\
\notag
\mu_3&=F_{xp}+pF_{yp}\big |_{(0,0,p)}
\end{align}

We have seen that $P_S=(0,0,\infty)$($q=0$) is always a critical point, and there $\mu_1=-\mu_3$.

The polynomial $\Pi(p)$ is a third order polynomial but the coefficient of $p^3$ is zero. The other critical values correspond to the zeros of  $\Pi_2(p)$, if they are real, where $\Pi_2(p)$ is $\Pi(p)$ seen as a second order polynomial:
\beq
\Pi_2(p)=\eta (\zeta_5 p^2+2\zeta_4 p+\zeta_3)
\eeq

We have:
\begin{align*}
\mu_1 =&-2\eta\left(\zeta_5p+\zeta_4\right)\\
\mu_3=&\eta\left(\zeta_4+\zeta_5p\right)
\end{align*}
and therefore:
\[
\mu_1=-2\mu_3
\]

Thus there are two possible types of critical points for the generic binary differential equation (\ref{bder}). The first case below corresponds to the existence of two extra singular points of the lifted vector field, that are necessarily saddles:

\begin{theorem}
For a generic Lagrangean surface, the phase portrait of the aymptotic lines around an inflection point is topologically conjugate to the phase portraits around the origin of:
\begin{itemize}
\item $Y \de Y^2 -2X\de X\de Y +Y^2 \de X^2=0, \quad$ (Fig.\ref{cusp1&3}, left)
\item $Y \de Y^2 +2X\de X\de Y +Y^2 \de X^2=0, \quad$ (Fig.\ref{cusp1&3}, right)
\end{itemize}
\end{theorem}

\begin{figure}
\begin{center}
\includegraphics[width= \linewidth]{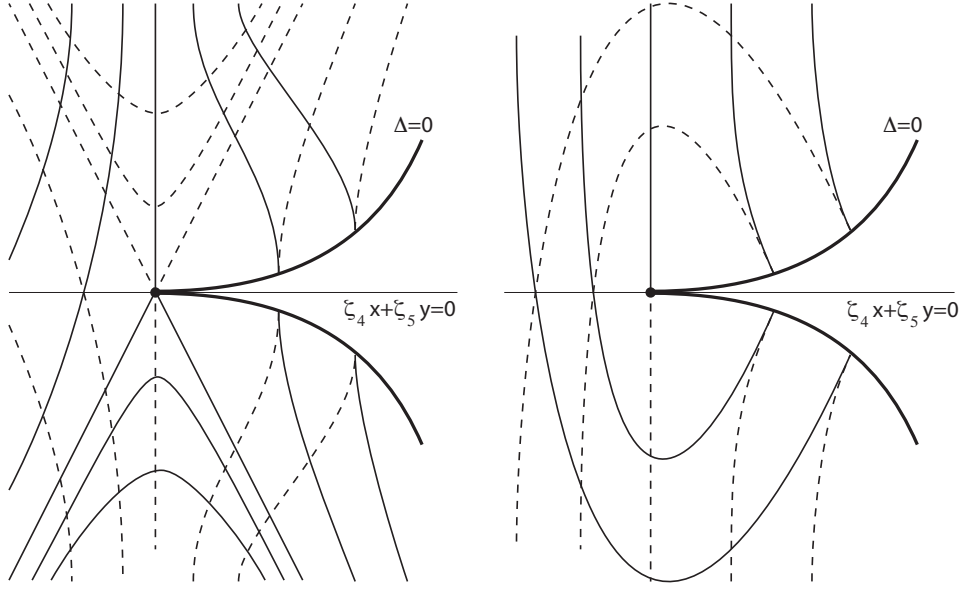}
\end{center}
\caption{Phase portraits for the asymptotic lines of generic Lagrangean surfaces}\label{cusp1&3}
\end{figure}

\begin{proof}
We follow the strategy of \cite{tari08, btjde}: we assume the 1-jet of the binary differential equation to be already in the form (\ref{nform1}) so that the 2-jet be:
\[
(x+a_2(x,y))\de y^2\pm 2(x+b_2(x,y) \de x \de y +c_2(x,y) \de x^2=0
\]
 and make a change of coordinates preserving the origin, whose linear part is the identity:
\[
x=X+P_2(X,Y), \quad y=Y+Q_2 (X,Y)
\]
with $P_2$and $Q_2$ homogeneous polynomials of degree 2, and we multiply the binary differential equation by a linear polynomial $1+R_1(X,Y)=1+r_1^0 X+r_1^1 Y$. The theorem is proved if we find the coefficients in $P_2$, $Q_2$ and $R_1$ so that the 2-jet of the new equation has the required form.

That 2-jet  after the change of variables becomes \cite{btjde}:
\[
(X+A_2(X,Y))\de Y^2\pm 2(X+B_2(X,Y)) \de X \de Y +C_2(X,Y) \de X^2=0
\]
where:
\begin{align*}
A_2=& a_2+Q_2\pm X\dfrac{\partial P_2}{\partial Y}+2Y\dfrac{\partial Q_2}{\partial Y}+YR_1\\
B_2=& b_2\pm P_2+Y\dfrac{\partial Q_2}{\partial X}\pm X\left(\dfrac{\partial P_2}{\partial X}+\dfrac{\partial Q_2}{\partial Y}+R_1\right)\\
C_2=&c_2\pm 2X\dfrac{\partial Q_2}{\partial X}
\end{align*}
The coefficient of $Y^2$ in $C_2$ is not changed (it can  later be made equal to 1 by a simultaneous rescaling of $X$ and $Y$), but the coefficients of $X^2$ and $XY$ vanish for a suitable choice of $q_2^0$ and $q_2^1$ in $Q_2 (X,Y)=q_2^0 X^2+q_2^1 XY +q_2^2 Y^2$.

The conditions:
\[
A_2=0, \quad B_2=0
\]
give six linear equations on the six remaining variables: the three coefficients of $P_2$, the two coefficients of $R_1$ and $q_2^2$. It is easy to see that the determinant of the system does not vanish.

\end{proof}

\begin{rem}
In general, there exists another type of critical point, when there are two extra singular points, a saddle and a node \cite{tari08}.
\end{rem}

\section{Global theory}

It is proved in \cite{little,asperti} that
if $S$ is a compact  surface with non vanishing Euler characteristic $\chi (S)$, then there exists at least  one inflection point or one umbilic point, and a line field on $S$ whose singularities are exactly the inflection and umbilic points. 

This can be improved for Lagrangean surfaces:

\begin{theorem}
Let $\mathcal L$ be a  compact and orientable Lagrangean surface with nonzero Euler characteristic $|\chi(\mathcal L)|$.  Then there exist at least an umbilic  point and an inflection point; in the generic case there are at least $|\chi(\mathcal L)|$ umbilic points, and at least $3|\chi(\mathcal L)|$ inflection points.
\end{theorem}
\begin{proof}
The identification $\Rq\cong \Cd$ given by $(z,w)\mapsto (x,y,u,v)$, where $z=x+iu$,  $w=y+iv$, allows the definition of a real operator $\mathbf J$ in $\Rq$ representing the multiplication by $i$ and given by:
\[
\mathbf J(x,y,u,v)=(-u,-v,x,y)
\]

We consider the isoclinic line field $\mathfrak h$ spanned by ${\mathbf J}^{-1}\mathcal H$  \cite{eu}.

\begin{lemma}
The  singularities of  the isoclinic line field $\mathfrak h$ on the Lagrangean surface $\mathcal L$  are exactly the umbilic points. 
\end{lemma}

\begin{proof}
All umbilic points are singularities of $\mathfrak h$, since at umbilic points we have $\mathcal H=0$; now if $p$ is a singularity, we have:

\begin{wineq}\cite{wintgen}
If $S$ is an immersed  surface in $\Rq$, then at every point $p\in S$ we have the inequality:
\beq
\label{wineq}
\mathcal H^2\ge K+|\kappa |
\eeq
The point $p$ is a circle point if and only if $\mathcal H^2 = K+|\kappa |$.
\end{wineq}

It follows  that, at $p$:
\[
K+|K|\le 0, \hbox{ \ and therefore \ } K+|K|= 0
\]
and so $p$ is a minimal point and a circle point, an umbilic point.

In particular, all minimal points are necessarily umbilic points in a Lagrangean surface, and at an umbilic the Gaussian and normal curvatures are nonpositive:
\[K=\kappa\le 0
\]
\end{proof}

In a generic situation the number of umbilic points is finite, and from the Poincar\'e-Hopf theorem it follows that:
\[
\sum_{\rm umbilics}\hbox{ind } \mathfrak h=\chi(\mathcal L)
\]
The estimate for the number of umbilic points follows from this  relation and from the fact that the indices of $\mathfrak h$ at generic critical points  are $\pm 1$.

The line through the mean curvature vector $\mathcal H (p)$ meets the ellipse of 
curvature $\mathcal E(p)$ at two points. The   unitary tangent vectors on $T_p S$ whose image by the second fundamental form is one of those two points span two orthogonal directions,  called  $\mathcal H$-directions. The mean directional field $\mathfrak H$ is the field of these two orthogonal directions.

The singularities of this 
field, called $\mathfrak H$-singularities, are the points where either $\mathcal H=0$ (minimal points) or at which the ellipse of curvature becomes a radial line segment 
(inflection points).

The differential equation of mean directional lines is given 
by:
\[
\sff(u)=m\mathcal H
\]
where $m\in \mathbf R$. Eliminating $m$  we have a binary differential 
equation:
\beq\label{bdemean}
A(x,y) \de x^2 + 2B(x,y) \de x \de y + C (x,y) \de y^2 = 0
\eeq
where: 
\begin{align*}
A =& A(x,y) = (ag-ce)E + 2(be-af)F\\
\notag
B = &B(x,y) = (bg-cf)E + (be-af)G\\
\notag
C =& C (x,y) = 2(bg-cf)F + (ce-ag)G
\end{align*}
The $\mathfrak H$-singularities are determined by $A = B = C = 0$. But 
it is immediate that the equation $EC = 2FB-GA$ holds, and the equation $C=0$ is redundant. 

For Lagrangean surfaces the binary differential equation (\ref{bdemean}) has the form:
\beq
A \de x^2+\
2B \de x \de y 
-A \de y^2 = 0
\eeq
with critical points at:
\[
A=EM-2FH_{\varphi}=0,\quad B=E^2H_{\psi}- EFM+F^2H_{\varphi}-H_{\varphi}=0
\]
Assuming the origin is a critical point and neglecting all second (and higher) order terms, as already done for the equation of the asymptotic directions, we obtain:

\begin{lemma}
The 1-jet of the binary differential equation (\ref{bdemean}) is the same as the 1-jet of the binary differential equation:
\beq\label{bdemdf}
M \de y^2+2(H_{\varphi}-H_{\psi})\de x\de y-M\de x^2=0
\eeq
\end{lemma}

\begin{lemma}
At a generic umbilic point in a Lagrangean surface, we have:
\[
\hbox{\rm ind } \mathfrak H=-\dfrac{1}{2}\hbox{\rm ind } \mathfrak h
\]
\end{lemma}
\begin{proof}
If the origin is a generic umbilic point, we have the normal form:
\begin{align*}
j^3\varphi=&\dfrac{1}{2}\alpha (x^2-y^2)+\beta xy+\dfrac{1}{3}\zeta_1x^3+\zeta_2x^2y+\zeta_3xy^2+\dfrac{1}{3}\zeta_4y^3\\
j^3\psi=&\dfrac{1}{2}\beta (x^2-y^2)-\alpha xy+\dfrac{1}{3}\zeta_2x^3+\zeta_3x^2y+\zeta_4xy^2 + \dfrac{1}{3}\zeta_5y^3
\end{align*}
with $\alpha^2+\beta^2\ne 0$.

The vector field:
\[
(\dot x, \dot y)=(a+c,e+g)=((\zeta_1+\zeta_3)x+(\zeta_2+\zeta_4)y, (\zeta_2+\zeta_4)x+(\zeta_3+\zeta_5)y)+O(2)
\]
spans the direction $ \mathfrak h$, and therefore $\hbox{\rm ind } \mathfrak h$ is $+1$ or $-1$ as the determinant:
\[
\left |
\begin{matrix}
\zeta_1+\zeta_3&\zeta_2+\zeta_4\\
\zeta_2+\zeta_4&\zeta_3+\zeta_5
\end{matrix}
\right |
\]
is positive or negative.

The vector fields:
\[
(\dot x, \dot y)=\left (M, -(H_{\varphi}-H_{\psi})\pm \sqrt{M^2+(H_{\varphi}-H_{\psi})^2}\right )\]
span the directions $ \mathfrak H$, and they rotate as the vector field 
$(\dot x, \dot y)=\left (M,-( H_{\varphi}-H_{\psi})\right )$ does. Since:
\begin{align*}
M=&(\alpha(\zeta_2+\zeta_4)-\beta(\zeta_1+\zeta_3))x+\\
&+(\alpha(\zeta_3+\zeta_5)-\beta(\zeta_2+\zeta_4))y+O(2)\\
H_{\varphi}-H_{\psi}=&-(\alpha(\zeta_1+\zeta_3)+\beta(\zeta_2+\zeta_4))x-\\
&(\alpha(\zeta_2+\zeta_4)+\beta(\zeta_3+\zeta_5))y+O(2)
\end{align*}
then $\hbox{\rm ind } \mathfrak H$ is $+1/2$ or $-1/2$ as the determinant:
\begin{align*}
&\left |
\begin{matrix}
\alpha(\zeta_2+\zeta_4)-\beta(\zeta_1+\zeta_3)&\alpha(\zeta_3+\zeta_5)-\beta(\zeta_2+\zeta_4)\\
\alpha(\zeta_1+\zeta_3)+\beta(\zeta_2+\zeta_4)&\alpha(\zeta_2+\zeta_4)+\beta(\zeta_3+\zeta_5)
\end{matrix}
\right |=
\\
&
=-(\alpha^2+\beta^2)
\left |
\begin{matrix}
\zeta_1+\zeta_3&\zeta_2+\zeta_4\\
\zeta_2+\zeta_4&\zeta_3+\zeta_5
\end{matrix}
\right |
\end{align*}
is positive or negative.
\end{proof}

To be precise, we should have considered the vector fields $(\dot x, \dot y)=\left (B,-A\pm \sqrt{A^2+B^2}\right )$ as spanning the directions defined by (\ref{bdemean}), but the same argument leads to the final vector field $(\dot x, \dot y)=\left (A,-B\right )$, and only its linear part  is relevant.

\begin{lemma}
At a generic inflection point in a Lagrangean surface, we have:
\begin{itemize}
\item If the topological model for the differential equation of the asymptotic lines is $Y \de Y^2 -2X\de X\de Y +Y^2 \de X^2=0$, then:
\[
\hbox{\rm ind } \mathfrak H=-\dfrac{1}{2}
\]
and the singularity of the mean directional field is of type $D_3$ (star) \cite{bfidal}.
\item If the topological model for the differential equation of the asymptotic lines is $Y \de Y^2 +2X\de X\de Y +Y^2 \de X^2=0$, then:
\[
\hbox{\rm ind } \mathfrak H=+\dfrac{1}{2}
\]
and the singularity of the mean directional field is of type $D_1$ (lemon) or $D_2$ (monstar)\cite{bfidal}.
\end{itemize}
\end{lemma}
\begin{proof}
Again $\hbox{\rm ind } \mathfrak H$ is $+1/2$ or $-1/2$ as the determinant of the linear part of $(\dot x, \dot y)=\left (M,-( H_{\varphi}-H_{\psi})\right )$ is positive or negative. From the normal form at inflection points it follows that:
\begin{align*}
M=&\eta (\zeta_4 x+\zeta_5y)+O(2)\\
H_{\varphi}-H_{\psi}=&\eta (\zeta_3 x+\zeta_4y)+O(2)
\end{align*}
then $\hbox{\rm ind } \mathfrak H$ is $+1/2$ or $-1/2$ as the determinant:
\[
-\eta^2
\left |
\begin{matrix}
\zeta_4&\zeta_5\\
\zeta_3&\zeta_4
\end{matrix}
\right |=-\eta^2(\zeta_4^2-\zeta_3\zeta_5)
\]
 is positive or negative.
 
 Recalling that for the binary differential equation of the asymptotic lines:
 \[
 \Pi(q)=\eta (\zeta_5 +2\zeta_4 q+\zeta_3 q^2) q
 \]
 we see that $\hbox{\rm ind } \mathfrak H$ is $+1/2$ or $-1/2$ as there exist one or three critical points for that equation.
\end{proof}

We already know that in the generic case:
\begin{itemize}
\item The singularities of $\mathfrak h$ are the umbilic points.
\item The indices of $\mathfrak h$ at those critical points are $\pm 1$.
\item The singularities of $\mathfrak H$ are the umbilic points and the inflection points.
\item The indices of $\mathfrak H$ at those critical points are $\pm 1/2$.
\item The  indices of $\mathfrak h$  and $\mathfrak H$ at umbilic points  have opposite signs ($\kappa<0$).
\end{itemize}
Note that if $\mathcal L$ is assumed to be  generic, we can prevent umbilic points where the curvature ellipse degenerates into a point ($K=\kappa=0$).

Then:
\[
\hbox{ind } \mathfrak H=-\dfrac{1}{2}\hbox{ind } \mathfrak h\hbox{ at umbilics, \ and thus \ }
\sum_{\rm umbilics}\hbox{ind } \mathfrak H=-\dfrac{1}{2}\chi(\mathcal L)
\]
and from:
\[
\sum_{\rm umbilics}\hbox{ind } \mathfrak H+\sum_{\rm inflections}\hbox{ind } \mathfrak H=\chi(\mathcal L)
\]
it follows:
\[
\sum_{\rm inflections}\hbox{ind } \mathfrak H=\dfrac{3}{2}\chi(\mathcal L)
\]
The estimate for the number of inflection points follows from the last  relation since the indices of $\mathfrak H$  are $\pm 1/2$. 
\end{proof}

\begin{rem}
From the previous proof, it follows that:
\begin{itemize}
\item The number of umbilic points is $\chi(\mathcal L)+2n$
\item The number of flat inflection points is $3\chi(\mathcal L)+2m$
\end{itemize}
with $n$ and $m$ nonnegative integers.
The minimal number of umbilic and inflection points happens when all critical points of $\mathfrak H$ have indices of the same sign as $\chi(\mathcal L)$.\end{rem}

\

\end{document}